\documentclass{amsart}
\usepackage{amsmath}
\usepackage{amssymb}
\usepackage{graphicx}
\usepackage{latexsym}
\usepackage{color}
\usepackage{amscd}
\usepackage[all]{xy}
\usepackage{enumerate}
\usepackage{enumitem}
\usepackage{hyperref}
\usepackage{subfigure}
\usepackage{soul}
\usepackage{comment}
\usepackage{epsfig}
\usepackage{epstopdf}
\usepackage{marginnote}
\usepackage{listings}
\usepackage{xcolor}

\usepackage{mparhack}
\usepackage{todonotes}
\usepackage[normalem]{ulem}

\usepackage{tikz}

\usepackage{multirow}

\newtheorem{thma}{Theorem}

\newtheorem{thmb}{Theorem}

\newtheorem{thmc}{Theorem}

\newtheorem{thm}{Theorem}
\newtheorem{lem}[thm]{Lemma}

\newtheorem{prop}[thm]{Proposition}

\newtheorem{theorem}[thm]{Theorem}

\newtheorem{proposition}[thm]{Proposition}
\newtheorem{question}[thm]{Question}

\theoremstyle{definition}

\newtheorem*{definition*}{Definition}

\newtheorem{remark}[thm]{Remark}

\newtheorem{example}[thm]{Example}

\newcommand{\CPb}{\overline{\mathbb{CP}}{}^{2}}
\newcommand{\CP}{{\mathbb{CP}}{}^{2}}

\newcommand{\RP}{{\mathbb{RP}}{}^{2}}
\newcommand{\R}{\mathbb{R}}

\newcommand{\Q}{\mathbb{Q}}

\newcommand{\Z}{\mathbb{Z}}
\newcommand{\SW}{{\rm SW}}

\newcommand{\twprod}{\mathbin{\mathchoice%
    {\ooalign{\raise1.15ex\hbox{$\scriptstyle\sim$}\cr\hidewidth$\times$\hidewidth\cr}}%
    {\ooalign{\raise1.15ex\hbox{$\scriptstyle\sim$}\cr\hidewidth$\times$\hidewidth\cr}}%
    {\ooalign{\raise.85ex\hbox{$\scriptscriptstyle\sim$}\cr\hidewidth$\scriptstyle\times$\hidewidth\cr}}%
    {\ooalign{\raise.65ex\hbox{$\scriptscriptstyle\sim$}\cr\hidewidth$\scriptscriptstyle\times$\hidewidth\cr}}%
    }}

\newcommand{\M}{\rm{Mod}}

\usepackage{scalerel} 

\def \x {\times}

\allowdisplaybreaks[4]

\begin{document}

\title[Exotic knottings and symmetries of surfaces in 4-manifolds] 
{Exotic knottings and symmetries of surfaces in 4-manifolds}

\author[R. \.{I}. Baykur]{R. \.{I}nan\c{c} Baykur}
\address{Department of Mathematics and Statistics, University of Massachusetts, Amherst, MA 01003, USA}
\email{inanc.baykur@umass.edu}

\author[N. Sunukjian]{Nathan Sunukjian}
\address{Department of Mathematics and Statistics,
Calvin University,
Grand Rapids, MI 49506}
\email{nss9@calvin.edu}


\begin{abstract}
We study exotic knottings of surfaces in $4$--manifolds through their
ambient symmetries.  We first give a general recipe for producing
projectively rigid surfaces, for which every smoothly extendable
self-diffeomorphism acts on first homology by $\pm I$.  For every
$g\geq 1$, a refinement of this construction yields a finite sequence
of genus-$g$ surfaces $F_0,\ldots,F_{2g}\subset X_g$,
which are topologically isotopic and topologically flexible---i.e.\ every
orientation-preserving self-diffeomorphism of $F_i$ can be realized by
a self-homeomorphism of $X_g$ preserving $F_i$---whereas successive
knotting rules out more and more projective homological symmetries, revealing a finer knottedness phenomenon.  The first two constructions combine iterated rim surgery with the convex geometry of Newton polytopes of relative Seiberg--Witten invariants. We also use hyperbolic geometry to construct a
totally geodesic surface of positive genus whose smooth and topological extendable
mapping class groups are both trivial.
\end{abstract}

\maketitle

\setcounter{secnumdepth}{2}
\setcounter{section}{0}



\section{Introduction}

Our viewpoint is reminiscent of Klein's Erlangen program
\cite{Klein1893}: we study exotic knottings of surfaces in $4$--manifolds through their ambient symmetries. Our guiding principle is that the loss of ambient symmetries reflects
more knottedness. Indeed, the
diffeomorphism and homeomorphism groups preserving an embedded surface
determine the smooth and topological surface and $4$--manifold pairs, respectively;
see Appendix~\ref{appendix:kleinwasright}. Since these complete
invariants are generally intractable, we compromise by passing to
the subgroups they induce in the better-understood mapping class group of the embedded surface.

Let $F$ be a smoothly embedded closed oriented surface in a closed oriented smooth $4$--manifold $X$.  A diffeomorphism
$\phi\in \rm{Diff}^+(F)$ is said to be \emph{smoothly extendable} if there is a diffeomorphism $\Phi \in \rm{Diff}^+(X)$ such that $\Phi|_F= \phi$.  It is similarly \emph{topologically extendable} if there is a
$\Phi\in\rm{Homeo}^+(X)$ such that $\Phi|_F=\phi$.
By isotopy extension, these yield the \emph{smooth} and \emph{topological extendable mapping class groups}  $\mathcal E^\infty(X,F)$ and $\mathcal E^0(X,F)$, such that
\[
 \mathcal E^\infty(X,F)
 \subseteq
 \mathcal E^0(X,F)
 \subseteq
 \M(F).
\]
These groups depend on the embedding $F\subset X$. We call $F$ smoothly, respectively topologically, \emph{flexible} when $\mathcal E^\infty(X,F)= \M(F)$ and $\mathcal E^0(X,F)= \M(F)$. Throughout the main text, ambient maps preserve orientation but need not be isotopic to the identity. 

The study of extendable mapping classes originated in the work of Montesinos and Iwase and was developed systematically by Hirose \cite{Montesinos1983,Iwase1988,Hirose1993,Hirose2002,Hirose2005,
HiroseYasuhara2008}.  The subject has continued to develop
\cite{DingLiuWangYao2012,LiuNiSunWang2013}, with a marked surge of
activity in recent years
\cite{WangWang2024,Salter2025,LawandeSaha2025,
BanerjeeSalter2025,Liu2026,Lehman2026,LehmanLewis,Niu2026, Pyronneau2026}. Our focus in this paper is on exotically knotted surfaces whose
topological extendable mapping class groups agree, while their smooth
ambient symmetries are increasingly constrained.  In this way,
extendable mapping class groups reveal that some exotic knottings are
more knotted than others.

At the opposite extreme from flexibility, we call $F\subset X$
(smoothly) \emph{rigid} if $\mathcal E^\infty(X,F)=\{1\}$, that is, if no nontrivial mapping class of $F$ is smoothly extendable. We call $F$ \emph{topologically rigid} if
$\mathcal E^0(X,F)=\{1\}$. For a surface of
positive genus, write
\[
 \rho\colon \M(F)\twoheadrightarrow
 \rm{PSp}\bigl(H_1(F;\Z)\bigr)
 :=
 \rm{Sp}\bigl(H_1(F;\Z)\bigr)/\{\pm I\}
\]
for the projectivized homological representation. Here $\rm{Sp}(H_1(F;\Z))$ denotes the group of symplectic lattice automorphisms preserving the algebraic intersection form on $F$.  We call $F$
\emph{projectively rigid} if
$\rho(\mathcal E^\infty(X,F))=\{1\}$.

Our first theorem gives a general construction of such surfaces.

\begin{thma}[Projective rigidity]
\label{thm:a}
Let $F\subset X$ be a smoothly embedded closed oriented surface of
genus $g\geq1$ in a closed simply connected oriented smooth
$4$--manifold.  Suppose that $F^2\geq0$,
$\pi_1(X\setminus\nu F)=1$, and $(X,F)$ has nonzero relative
Seiberg--Witten invariant.  Then $F$ admits a topologically isotopic,
projectively rigid exotic copy $F'\subset X$.
\end{thma}

The surface $F'$ is obtained from $F$ by $2g$ carefully chosen iterated rim surgeries. If $F$ is ordinary (that is, noncharacteristic), then simple connectivity of the complement implies that $F'$ is topologically flexible; see Remark~\ref{rk:top-flexible}. When $g=1$, our result
comes within one mapping class of smooth rigidity: the only possible
nontrivial smoothly extendable mapping class is the hyperelliptic
involution.  Variants of this construction for surfaces with finite
cyclic complement group, using twist rim surgery, are given in
Remarks~\ref{rmk:cyclic-rim} and~\ref{rmk:plane-curves}. 

\begin{thmb}[Successive symmetry breaking]
\label{thm:b}
For every $g\geq1$, there exist a closed simply connected oriented
smooth $4$--manifold $X_g$ and smoothly embedded closed oriented
genus-$g$ surfaces
\[
 F_0,F_1,\ldots,F_{2g}\subset X_g
\]
such that $F_0$ is smoothly flexible, all the $F_i$ are mutually
topologically isotopic and topologically flexible, and the pairs
$(X_g,F_i)$ are pairwise nondiffeomorphic. After choosing identifications $F_i\cong\Sigma_g$, there is an integral
basis $\{v_1,\ldots,v_{2g}\}$ of $H_1(\Sigma_g;\Z)$ such that, if $P_i$
denotes the subgroup of $\rm{PSp}(2g,\Z)$ preserving each of the
primitive lines $\{\Z v_1,\ldots,\Z v_i\}$,
with $P_0=\rm{PSp}(2g,\Z)$, we have, for every $i=0,\ldots,2g$,
\[
 P_0\supsetneq P_1\supsetneq\cdots\supsetneq P_{2g}=\{1\},
 \qquad
 \rho\bigl(\mathcal E^\infty(X_g,F_i)\bigr)\subseteq P_i.
\]
\end{thmb}

Thus each rim surgery forces one additional primitive homology line
to be preserved. The strict filtration is the sequence of upper
bounds $P_i$; we do not assert that the actual images
$\rho(\mathcal E^\infty(X_g,F_i))$ are nested.  This is an exotic
analogue of the loss of symmetry already visible in the surface-knot
examples of Iwase and Hirose \cite{Iwase1988,Hirose1993}.

A complementary application of hyperbolic geometry gives a surface that is rigid in both categories:

\begin{thmc}[A completely rigid example]
\label{thm:c}
There is a closed oriented hyperbolic $4$--manifold $X$ containing a
closed oriented embedded totally geodesic surface $F$ of genus $40$
such that
\[
 \mathcal E^\infty(X,F)
 =
 \mathcal E^0(X,F)
 =
 \{1\}.
\]
\end{thmc}

Such an example is a natural candidate for a surface admitting
no exotic copies; see also Section~\ref{sec:questions}.

We prove Theorems~\ref{thm:a} and~\ref{thm:b} in
Sections~\ref{sec:rigid} and~\ref{sec:filtration}.  The proofs use the
relative Seiberg--Witten rim-surgery formula of Fintushel and Stern
\cite{FintushelStern1997,FintushelSternAddendum}.  Inspired by
McMullen--Taubes and Vidussi
\cite{McMullenTaubes1999,Vidussi2001}, we develop a convex-geometric
package for the relative invariant, adapted to rim surgery, and combine
it with integral symplectic linear algebra.  From the support of the
invariant we extract a finite collection of translation classes of
Newton polytopes, which we call its \emph{rim Newton profile}; this is
insensitive to the natural monomial and sign ambiguities.  Under rim
surgery, each polytope changes by Minkowski addition of a weighted
segment; under iteration, the added segments form zonotopes whose
geometry distinguishes the resulting pairs, while their symplectic
stabilizers constrain the smooth ambient symmetries.  In particular, a positive-dimensional rim Newton
polytope obstructs smooth flexibility; see
Proposition~\ref{prop:flexible-profile}.  Reciprocity of the Alexander
polynomials appearing in the rim-surgery formula,
$\Delta_K(t^{-1})=\Delta_K(t)$ in the symmetrized normalization, makes
the added zonotopes centrally symmetric.  Thus the method cannot
distinguish a rim class from its negative, leaving the residual $\pm I$
ambiguity.  

We prove Theorem~\ref{thm:c} in Section~\ref{sec:hyperbolic}. The proof combines Mostow rigidity and the
Dehn--Nielsen--Baer theorem with arithmetic embedding results for
hyperbolic manifolds. Section~\ref{sec:questions} concludes with some questions.

In Appendix~\ref{appendix:kleinwasright}, we show that the relative
diffeomorphism and homeomorphism groups of a pair $(X,F)$ determine the
pair up to diffeomorphism and homeomorphism, respectively.
Proposition~\ref{prop:mcg-forgets} illustrates how passing to extendable
mapping class groups loses information.


\smallskip
\section{Projectively rigid surfaces}
\label{sec:rigid}

The proof of Theorem~\ref{thm:a} combines the naturality of extendable
mapping class groups, the Fintushel--Stern rim-surgery formula, and a
Newton-polytope argument for the relative Seiberg--Witten invariants. We begin with some basic observations about
extendability and projective homology.

\subsection{Extendable groups and projective homology}

\begin{lem}
\label{lem:extendable-conjugacy}
Let
$h\colon (X,F)\to (X',F')$
be an orientation-preserving homeomorphism of oriented pairs.
Conjugation by $h|_F$ gives an isomorphism
$\M(F)\to\M(F')$ carrying
$\mathcal E^0(X,F)$ onto $\mathcal E^0(X',F')$.
The analogous statement holds in the smooth category when $h$ is a
diffeomorphism.
\end{lem}

\begin{proof}
If $[\phi]\in\mathcal E^0(X,F)$ is induced by
$\Phi\in\rm{Homeo}^+(X)$, then $h\Phi h^{-1}$ preserves $F'$ and
induces
\[
 h|_F\phi(h|_F)^{-1}
\]
on $F'$.  Applying the same argument to $h^{-1}$ gives the reverse
inclusion.
\end{proof}

Thus, under the identification induced by a topological isotopy,
topologically isotopic surfaces have the same topological extendable
mapping class group.  In particular, as we have the natural inclusion
$\mathcal E^\infty(X,F)\subseteq\mathcal E^0(X,F)$, a smoothly
flexible surface is topologically flexible, as is every surface
topologically isotopic to it.

For a closed oriented surface $F$ of positive genus, the action of
$\M(F)$ on $H_1(F;\Z)$ preserves the algebraic intersection form.  We
write
\[
 \rho\colon\M(F)\rightarrow
 \rm{PSp}\bigl(H_1(F;\Z)\bigr)
 :=
 \rm{Sp}\bigl(H_1(F;\Z)\bigr)/\{\pm I\},
 \qquad
 \rho(\phi)=[\phi_*],
\]
for the projectivized homological representation.  Thus
\[
 \ker\rho
 =
 \left\{
 \phi\in\M(F)\mid
 \phi_*=\pm I\text{ on }H_1(F;\Z)
 \right\}.
\]
Note that the Torelli group
\[
 \mathcal I(F)
 =
 \left\{
 \phi\in\M(F)\mid\phi_*=I
 \right\}
\]
is an index-$2$ subgroup of $\ker\rho$. When $F=T^2$,
\[
 \M(T^2)\cong\rm{SL}(2,\Z),
 \qquad
 \mathcal I(T^2)=\{1\},
 \qquad
 \ker\rho=\{\pm I\},
\]
where $-I$ is the unique nontrivial central mapping class, represented by the hyperelliptic involution.

\subsection{Rim surgery and relative Seiberg--Witten invariants}

We now recall the Fintushel--Stern rim-surgery construction and its effect on the relative invariant. Assume first that $F^2=0$, and set
\[
 M_F=X\setminus\rm{int}\nu F.
\]
Fix a trivialization $\nu F\cong F\times D^2$, so that
$\partial M_F\cong F\times S^1$.  For an oriented simple closed curve
$\gamma\subset F$, let
\[
 R_\gamma=\gamma\times\mu_F\subset\partial M_F
\]
be the corresponding \emph{rim torus}, where $\mu_F$ is a positively
oriented normal circle to $F$.  Under the assumption
$\pi_1(M_F)=1$, the rim tori generate a distinguished subgroup  $\mathcal R_F \subset H_2(M_F;\Z)$, which can be identified as
\[
 \mathcal R_F\cong H_1(F;\Z).
\]
We use the
same notation $\mathcal R_F$ for this identified lattice, labeling its
elements by the corresponding classes $[R_\gamma]$. 
A diffeomorphism of pairs restricting to $\phi$ on $F$ sends
$[R_\gamma]$ to $[R_{\phi(\gamma)}]$, and hence acts on
$\mathcal R_F$ as $\phi_*$.

We recall the relative Seiberg--Witten invariant of Fintushel and
Stern \cite{FintushelSternAddendum}; for general background and
conventions, see also \cite{FintushelSternLectures}.  By the
K\"unneth formula,
\[
 H^2(F\times S^1;\Z)
 \cong
 H^2(F;\Z)\oplus H^1(F;\Z)
 \cong
 \Z\oplus H^1(F;\Z),
\]
with no torsion terms.  Since $F\times S^1$ is spin and $H^2(F\times S^1;\Z)$ has no
$2$--torsion, the first Chern class identifies ${\rm Spin}^c$
structures on $F\times S^1$ with the even cohomology classes; i.e.\
we identify $\mathfrak s$ with
$\frac12c_1(\mathfrak s)\in H^2(F\times S^1;\Z)$.

Following \cite{FintushelStern1997, FintushelSternAddendum}, let $\mathfrak s_k$ denote the
${\rm Spin}^c$ structure corresponding under this identification to
\[
 (k,0)\in H^2(F;\Z)\oplus H^1(F;\Z).
\]
Thus
\[
 c_1(\mathfrak s_k)=2(k,0).
\]
If $u\in H^2(F;\Z)$ is the positive generator satisfying
$\langle u,[F]\rangle=1$, then $(k,0)$ means $(ku,0)$, and hence
\[
 \bigl\langle c_1(\mathfrak s_{g-1}),[F]\bigr\rangle
 =
 2g-2.
\]

Let $\mathcal T_F$ be the set of ${\rm Spin}^c$ structures on
$M_F$ whose boundary value is $\mathfrak s_{g-1}$ or
$-\mathfrak s_{g-1}$, where $-\mathfrak s_{g-1}$ denotes the
conjugate ${\rm Spin}^c$ structure.  Thus $\mathcal T_F$ is the union
of one or two affine pieces, according as these boundary structures
coincide or not.  Set
\[
 \Lambda_F
 =
 \ker\left(
 H^2(M_F;\Z)\rightarrow H^2(\partial M_F;\Z)
 \right).
\]
Each nonempty affine piece of $\mathcal T_F$ is a torsor over
$\Lambda_F$: twisting an extension by an element of $\Lambda_F$
preserves its boundary value, and the difference of any two
extensions with the same boundary value lies in $\Lambda_F$.  There
is no preferred origin.  For $g=1$ the two boundary structures
coincide, while for $g>1$ they are distinct.

Since
$\pi_1(M_F)=1$, the group $H^2(M_F;\Z)$ has no $2$--torsion, so the
first Chern class is injective on $\mathcal T_F$.  We therefore use
\[
 \mathcal A_F
 =
 \{c_1(\tau)\mid \tau\in\mathcal T_F\}
\]
as the affine exponent set. Each affine piece of $\mathcal A_F$ is affine over $2\Lambda_F$.
Through the cohomological identification used in the
Fintushel--Stern rim-surgery formula, the homological rim-torus lattice
$\mathcal R_F\cong H_1(F;\Z)$ determines a distinguished lattice of
translation directions in $\Lambda_F$.  We use the same notation
$\mathcal R_F$ for this identified copy.

After choosing generators of the relevant rank-one monopole Floer
homology groups, the relative Seiberg--Witten invariant becomes a
finitely supported integer-valued function on $\mathcal T_F$.  Using the first
Chern class, we encode it as the affine Laurent polynomial
\[
 f_{X,F}
 =
 \sum_{\alpha\in\mathcal A_F}
 a_\alpha e^\alpha , 
\]
where $a_\alpha\in\Z$.
Choosing origins in the affine components identifies this with one or
two ordinary Laurent polynomials.  Changing an origin multiplies the
corresponding polynomial by a monomial, while changing the Floer generators changes the signs of the corresponding
polynomials.  We write $\doteq$ for equality
up to these ambiguities.

Let $F_{K,\gamma}\subset X$ denote the result of rim surgery on $F$
along $\gamma$ using a knot $K\subset S^3$.

\begin{prop}[Fintushel--Stern
\cite{FintushelStern1997,FintushelSternAddendum}]
\label{prop:FS-rim-formula}
Suppose that $F^2=0$ and $\pi_1(M_F)=1$.  Then
\[
 \pi_1(X\setminus\nu F_{K,\gamma})=1,
\]
and there are natural identifications of the affine exponent sets and
rim-torus subgroups for $(X,F)$ and $(X,F_{K,\gamma})$.  Under these
identifications,
\[
 f_{X,F_{K,\gamma}}
 \doteq
 f_{X,F}\,\Delta_K(r_\gamma^2),
 \tag{2.1}
\]
where $r_\gamma=e^{[R_\gamma]}$, so that
$r_\gamma^2=e^{2[R_\gamma]}$.  Moreover, a diffeomorphism of pairs
acts naturally on the affine exponent set, preserves the relative
invariant up to the above ambiguities, and acts on the rim-torus
subgroup through its action on $H_1(F;\Z)$.
\end{prop}

Formula~(2.1) can be iterated.  Let
\[
 F_0=F,\qquad
 F_i=(F_{i-1})_{K_i,\gamma_i},
 \qquad i=1,\ldots,m,
\]
where, under the natural identifications after each surgery,
$\gamma_i$ represents $v_i\in H_1(F;\Z)$.  Then
\[
 f_{X,F_m}
 \doteq
 f_{X,F}
 \prod_{i=1}^m\Delta_{K_i}(r_{v_i}^2).
 \tag{2.2}
\]
In particular, $f_{X,F_m}\neq0$ whenever $f_{X,F}\neq0$.

The topology of the resulting surfaces is also unchanged in the
strong sense needed here.

\begin{lem}
\label{lem:rim-topology}
Suppose that $X$ is simply connected and
$\pi_1(X\setminus\nu F)=1$.  Every surface obtained from $F$ by a
finite sequence of rim surgeries is topologically isotopic to $F$.
Consequently, under the identification induced by this isotopy, $
 \mathcal E^0(X,F_i)=\mathcal E^0(X,F)
$ 
for every $i$.
\end{lem}

\begin{proof}
Rim surgery preserves the genus, homology class, and simple
connectivity of the complement
\cite{FintushelStern1997}.  The surfaces are therefore topologically
isotopic by Boyer's theorem
\cite[Theorem~F]{Boyer1993}; see also
\cite[Theorem~7.1]{Sunukjian2015}.  The assertion about extendable
groups follows from
Lemma~\ref{lem:extendable-conjugacy}.
\end{proof}

\begin{remark}
\label{rmk:cyclic-rim}
The preceding rim-surgery results have analogues for surfaces with
finite cyclic complement group, using \emph{twist rim surgery}; see
\cite{Kim2006} for the definition.  As shown in the proof of
\cite[Theorem~3.4]{Kim2006}, twist rim surgery has the same effect on
the relative Seiberg--Witten invariant as ordinary rim surgery.
Under the hypotheses of
\cite[Propositions~2.3 and~2.4]{KimRubermanSmooth}, a $1$--twist
preserves the complement group, and more generally an $m$--twist
preserves a cyclic complement group $\Z_d$ when $(m,d)=1$.  Under the
hypotheses of \cite[Theorem~1.3]{KimRubermanTopological}, the
resulting surfaces are topologically isotopic to the original
surface.  We use only the $1$--twist case in
Remark~\ref{rmk:plane-curves}.
\end{remark}

If $F^2=n>0$, the Relative Seiberg-Witten invariant is most simply defined by reducing to the previous case through blowups. Specifically, let
\[
 \widetilde X=X\#n\CPb
\]
and let $\widetilde F\subset\widetilde X$ be the proper transform
obtained by blowing up $n$ points of $F$.  Thus
$\widetilde F^2=0$.  Throughout the paper, by the relative
Seiberg--Witten invariant of $(X,F)$ we mean
\[
 \SW^{\rm rel}_{X,F}
 :=
 \SW^{\rm rel}_{\widetilde X,\widetilde F}.
\]
Rim surgeries will always be performed away from the blow-up points.
This is the convention used in the nonnegative self-intersection
setting of Fintushel and Stern.

\subsection{Rim Newton polytopes}

We now extract convex-geometric information from the support of the
relative Seiberg--Witten invariant. We only use standard facts on 
Newton polytopes, Minkowski sums, and zonotopes; for more background see e.g.
\cite{Sturmfels1996,Ziegler1995}.

Let $R$ be a rank-$n$ lattice.  A choice of basis identifies its group ring with
a Laurent polynomial ring:
\[
 \Z[R]\cong
 \Z[x_1^{\pm1},\ldots,x_n^{\pm1}].
\]
Let
\[
 p=\sum_{u\in R}a_u e^u\in\Z[R]
\]
be a nonzero Laurent polynomial.  
Here $e^u$ denotes the formal group-ring element indexed by $u$; after choosing a basis of $R$, it becomes the corresponding Laurent monomial. 
Its \emph{support} and \emph{Newton polytope} are
\[
 {\rm supp}(p)=\{u\in R\mid a_u\neq0\},
 \qquad
 {\rm Newt}(p)
 =
 {\rm conv}\bigl({\rm supp}(p)\bigr)
 \subset R\otimes\R.
\]
Here ${\rm conv}$ denotes the convex hull.
Since ${\rm supp}(p)$ is finite, its convex hull is a polytope.
For two polytopes $P,Q\subset R\otimes\R$, their \emph{Minkowski sum} is
\[
 P+Q=\{x+y\mid x\in P,\ y\in Q\}.
\]
If $p,q\in\Z[R]$ are nonzero, then the standard product rule gives
\[
 {\rm Newt}(pq)
 =
 {\rm Newt}(p)+{\rm Newt}(q).
 \tag{2.3}
\]

A Minkowski sum of finitely many line segments is called a
\emph{zonotope}.

We apply this to the rim torus lattice
\[
 R=\mathcal R_F\cong H_1(F;\Z).
\]
The affine exponent set $\mathcal A_F$ carries a free translation
action by $2R$. For each orbit
\[
 \lambda\in\mathcal A_F/(2R)
\]
which meets the support of $f_{X,F}$, choose
$\alpha_\lambda\in\lambda$ and write
\[
 f_\lambda
 =
 \sum_{r\in R}
 a_{\alpha_\lambda+2r}e^{2r}
 \in\Z[2R].
\]
Equivalently, the portion of $f_{X,F}$ supported on $\lambda$ is
$e^{\alpha_\lambda}f_\lambda$; thus $f_\lambda$ records that orbit after translating $\alpha_\lambda$ to the origin. See Example~\ref{ex:g1-rim-newton} in the next section for a 
genus $1$ illustration of the polynomials $f_\lambda$ and their Newton
polytopes.

Changing $\alpha_\lambda$ to $\alpha_\lambda+2s$, for $s\in R$,
multiplies $f_\lambda$ by the monomial $e^{-2s}$.  Thus the Newton
polytope itself depends on the chosen affine origin only by
translation, and its translation class
\[
 \bigl[{\rm Newt}(f_\lambda)\bigr]
\]
is well defined.

We call the finite collection
\[
 \mathcal N_R(f_{X,F})
 =
 \left\{
 \bigl[{\rm Newt}(f_\lambda)\bigr]
 \ \middle|\
 \lambda\in\mathcal A_F/(2R),\ f_\lambda\neq0
 \right\}
\]
the \emph{rim Newton profile} of $(X,F)$.  Here $[P]$ denotes the translation class of a polytope $P$, and
repeated classes are retained with their multiplicities.

\begin{lem}
\label{lem:rim-Newton-profile}
Let
\[
 \Phi\colon(X,F)\rightarrow(X',F')
\]
be a diffeomorphism of pairs, and let
$A\colon\mathcal R_F\to\mathcal R_{F'}$
be the induced lattice isomorphism.  Then
\[
 A\bigl(\mathcal N_{\mathcal R_F}(f_{X,F})\bigr)
 =
 \mathcal N_{\mathcal R_{F'}}(f_{X',F'}).
\]
Here $A$ acts on translation classes through its real-linear
extension, so that $A([P])=[A(P)]$.

Moreover, if $0\neq q\in\Z[2R]$, then
\[
 \mathcal N_R(f_{X,F}q)
 =
 \left\{
 [P+{\rm Newt}(q)]
 \ \middle|\
 [P]\in\mathcal N_R(f_{X,F})
 \right\}.
\]
\end{lem}

\begin{proof}
Naturality of the relative invariant shows that $\Phi$ induces an
affine bijection of the exponent sets.  If two exponents differ by
$2r$, their images differ by $2A(r)$, so this bijection permutes the
$2R$--orbits.  Within one orbit, $f_\lambda$ records the differences
of the exponents from the chosen origin.  Under $\Phi$, these
differences are transformed by $A$, while changing the origin only
multiplies the polynomial by a monomial.  Thus
\[
 \bigl[{\rm Newt}(f_\lambda)\bigr]
 \longmapsto
 \bigl[A({\rm Newt}(f_\lambda))\bigr].
\]
The orbits may be permuted, but the resulting multiset of translation
classes is preserved.

Since the exponents of $q$ lie in $2R$, multiplication by $q$ does not
mix the different orbits in $\mathcal A_F/(2R)$; on the orbit
$\lambda$, the new polynomial is $f_\lambda q$.  Since $R$ is a
lattice, $\Z[2R]$ is a Laurent polynomial ring and hence an integral
domain.  Thus
\[
 f_\lambda q\neq0
\]
whenever $f_\lambda\neq0$ and $q\neq0$.  Equation~(2.3) gives
\[
 {\rm Newt}(f_\lambda q)
 =
 {\rm Newt}(f_\lambda)+{\rm Newt}(q),
\]
which proves the second assertion, with multiplicities preserved.
\end{proof}

\begin{remark}
\label{rk:newton-forgets}
Although we do not need this refinement here, retaining the collection
of polynomials $f_\lambda$, up to the monomial and sign ambiguities
above, gives a finer invariant whose stabilizer may be smaller than
that of the rim Newton profile, which records only the translation
classes of their Newton polytopes and therefore forgets their
coefficients.
\end{remark}

The zonotope contributed by the rim-surgery factors is centrally
symmetric, so both $I$ and $-I$ preserve it.  Thus its symplectic
stabilizer necessarily contains $\{\pm I\}$.  We now choose the directions and weights so that this is the full stabilizer: distinct
weights prevent permutations of the generating directions, while a
connected intersection graph forces all remaining sign choices to
agree.

\begin{lem}
\label{lem:rigid-zonotope}
For every $g\geq1$, there is an integral basis
$\{v_1,\ldots,v_{2g}\}$ of $H_1(\Sigma_g;\Z)$ and pairwise distinct
positive integers $d_1,\ldots,d_{2g}$ such that the stabilizer in
$\rm{Sp}(2g,\Z)$ of the zonotope
\[
 Z
 =
 \sum_{i=1}^{2g}
 [-d_iv_i,d_iv_i]
\]
is $\{\pm I\}$.
\end{lem}

\begin{proof} 
Let
\[
 V=H_1(\Sigma_g;\Z),
\]
viewed as a rank-$2g$ symplectic lattice, and let
$\{a_1,b_1,\ldots,a_g,b_g\}$ be a symplectic basis.  Set
\[
 v_{2i-1}=a_i
 \quad (1\leq i\leq g),
 \qquad
 v_{2i}=b_i+b_{i+1}
 \quad (1\leq i<g),
 \qquad
 v_{2g}=b_g.
\]
These classes form an integral basis: we have $a_i=v_{2i-1}$ and,
starting with $b_g=v_{2g}$, recover
$b_i=v_{2i}-b_{i+1}$ successively.

Their intersection graph $\Gamma$, in which two vertices are joined when their
algebraic intersection is nonzero, is connected.  Indeed, for
$1\leq i<g$,
\[
 v_{2i-1}\cdot v_{2i}\neq0,
 \qquad
 v_{2i}\cdot v_{2i+1}\neq0,
\]
and moreover $v_{2g-1}\cdot v_{2g}\neq0$.  Thus $\Gamma$ contains the
path
\[
 v_1-v_2-\cdots-v_{2g}.
\]

Choose $
 0<d_1<d_2<\cdots<d_{2g}$. 
Since the $v_i$ form a basis of $V$, the induced linear isomorphism $
 \R^{2g}\rightarrow V\otimes\R
$
sending the standard basis to the $v_i$ carries the box 
\[
\prod_{i=1}^{2g}[-d_i,d_i]
\]
onto $Z$. So the edges of $Z$ are parallel to the lines
$\R v_i$, and a parallel edge to $\R v_i$ has lattice length $2d_i$.

Let $A\in\rm{Sp}(2g,\Z)$ preserve $Z$.  It sends edges to edges and,
as an integral lattice automorphism, preserves their lattice lengths. Here the lattice length of an edge with edge vector $kv$, where $v$
is primitive, is $|k|$.
Since the $2d_i$ are pairwise distinct, $A$ preserves each line
$\R v_i$.  Since $v_i$ is primitive, it follows that
\[
 A(v_i)=\varepsilon_i v_i,
 \qquad
 \varepsilon_i\in\{\pm1\}.
\]
Now suppose $v_i\cdot v_j\neq0$.  Since $A$ is symplectic, it
preserves the intersection pairing, and hence
\[
 v_i\cdot v_j
 =
 A(v_i)\cdot A(v_j)
 =
 \varepsilon_i\varepsilon_j(v_i\cdot v_j).
\]
Thus $\varepsilon_i\varepsilon_j=1$, so
$\varepsilon_i=\varepsilon_j$.  The signs therefore agree along every
edge of the intersection graph $\Gamma$.  Since $\Gamma$ is connected, they
all agree, and consequently $A=I$ or $A=-I$.
\end{proof}

Since each segment $[-d_iv_i,d_iv_i]$ is centered at the origin, the
zonotope $Z$ is centrally symmetric, so $Z=-Z$.  Since the $v_i$ form
a basis, $Z$ is also full dimensional.  In the next lemma, $Z$ denotes
an arbitrary centrally symmetric full-dimensional lattice polytope,
not necessarily the particular zonotope constructed above; in our
application it will be that zonotope.  

Let $R$ be a rank-$n$ lattice and set
\[
 {\rm GL}(R):={\rm Aut}_{\Z}(R);
\]
after choosing a basis, ${\rm GL}(R)\cong{\rm GL}(n,\Z)$.  The next
lemma shows that, for sufficiently large $m$, any lattice automorphism
preserving the translation classes of the polytopes $P+mZ$ must
preserve $Z$ itself.  Thus the added zonotope remains detectable in
the presence of the Newton polytopes coming from the original relative
invariant.

\begin{lem}
\label{lem:large-zonotope}
Let $R$ be a rank-$n$ lattice, let $\mathcal P$ be a nonempty finite
collection of lattice polytopes in $R\otimes\R$, and let $Z=-Z$ be a
full-dimensional lattice polytope.  For all sufficiently large
positive integers $m$, every $A\in{\rm GL}(R)$ preserving the
collection of translation classes
\[
 \left\{
 [P+mZ]\mid P\in\mathcal P
 \right\}
\]
satisfies $A(Z)=Z$.
\end{lem}

\noindent
Here preserving the collection means preserving it setwise: for every
$P\in\mathcal P$, there is some $Q\in\mathcal P$ such that
$A(P+mZ)$ is a translate of $Q+mZ$.

\begin{proof}
Suppose otherwise.  Then there are positive integers $m_k\to\infty$
and $A_k\in{\rm GL}(R)$ preserving these collections but satisfying
$A_k(Z)\neq Z$.  Fix $P\in\mathcal P$.  For every $k$, there are
$Q_k\in\mathcal P$ and $t_k\in R\otimes\R$ such that
\[
 A_k(P+m_kZ)=Q_k+m_kZ+t_k.
\]
Since $\mathcal P$ is finite, some $Q\in\mathcal P$ occurs as $Q_k$
for infinitely many $k$, by the pigeonhole principle.  After passing to this subsequence, we have
\[
 A_k(P+m_kZ)=Q+m_kZ+t_k.
 \tag{2.4}
\]
For a polytope $S\subset R\otimes\R$, let
$\Delta(S):=S-S$ denote its \emph{difference body}.  This operation is
translation invariant, equivariant under linear maps, and additive
under Minkowski sums. Taking difference bodies in Equation~(2.4), and using translation
invariance, equivariance under $A_k$, and Minkowski additivity, gives
\[
 A_k\bigl(\Delta(P)+2m_kZ\bigr)
 =
 \Delta(Q)+2m_kZ.
 \tag{2.5}
 \label{eq:difference-body}
\]
where we used $Z=-Z$.

Since $Z=-Z$ is full dimensional, the origin lies in the interior of
$Z$.  Each $\Delta(P')$ is compact, and $\mathcal P$ is finite, so
there is a constant $C>0$ such that
\[
 \Delta(P')\subset CZ
 \qquad\text{for every }P'\in\mathcal P.
\]
Since $0\in\Delta(P)$, we have
\[
 2m_kZ\subset\Delta(P)+2m_kZ.
\]
Applying $A_k$ and using Equation~\eqref{eq:difference-body}, we obtain
\[
 A_k(2m_kZ)
 \subset
 \Delta(Q)+2m_kZ
 \subset
 CZ+2m_kZ
 =
 (2m_k+C)Z.
\]
Equivalently,
\[
 A_k(Z)\subset
 \left(1+\frac{C}{2m_k}\right)Z.
 \tag{2.6}
 \label{eq:Ak-bound}
\]

Once again, since $Z=-Z$ is a full-dimensional convex body, its \emph{Minkowski
functional}
\[
 \|x\|_Z:=\inf\{t>0\mid x\in tZ\}
\]
is a norm whose unit ball is $Z$; see
\cite{Ball1997}. Let
\[
 \|A\|_{{\rm op},Z}
 :=
 \sup_{\|x\|_Z\leq1}\|A(x)\|_Z
\]
denote the corresponding operator norm.  Equation~\eqref{eq:Ak-bound}
gives
\[
 \|A_k\|_{{\rm op},Z}
 \leq
 1+\frac{C}{2m_k},
\]
so these operator norms are uniformly bounded.

Choose a basis of $R$ and represent each $A_k$ by an integral
$n\times n$ matrix.  By equivalence of norms on the finite-dimensional
space $R\otimes\R$, the preceding bound gives a uniform bound on the
Euclidean operator norms of these matrices, and hence on all their
entries.  Since the entries are integers, only finitely many matrices
can occur.  One of them therefore occurs for infinitely many $k$;
passing to the corresponding subsequence, which still satisfies
$m_k\to\infty$, we may assume that $A_k=A$ is constant.

Dividing Equation~\eqref{eq:difference-body} by $2m_k$ gives
\[
 A\left(
 Z+\frac{1}{2m_k}\Delta(P)
 \right)
 =
 Z+\frac{1}{2m_k}\Delta(Q).
\]
Since $\Delta(P)$ and $\Delta(Q)$ are bounded, the scaled difference
bodies shrink uniformly to $\{0\}$ as $m_k\to\infty$.  Thus the two
sides converge to $A(Z)$ and $Z$, respectively, and hence $A(Z)=Z$,
contradicting the choice of the $A_k$.
\end{proof}

\subsection{Proof of Theorem~\ref{thm:a}}

We prove the following slightly stronger form of Theorem~\ref{thm:a}.

\begin{thm}
\label{thm:projective-rigidity-detailed}
Let $F\subset X$ satisfy the hypotheses of
Theorem~\ref{thm:a}.  There are primitive classes
$v_1,\ldots,v_{2g}$ forming an integral basis of $H_1(F;\Z)$,
simple closed curves $\gamma_i\subset F$ representing them, and knots
$K_i\subset S^3$ such that, setting
\[
 F_0=F,
 \qquad
 F_i=(F_{i-1})_{K_i,\gamma_i}
 \quad (i=1,\ldots,2g),
 \qquad
 F'=F_{2g},
\]
the following hold:
\begin{enumerate}
\item
$F'$ is topologically isotopic to $F$, and under this identification
\[
 \mathcal E^0(X,F')=\mathcal E^0(X,F);
\]

\item
the pairs $(X,F)$ and $(X,F')$ are not diffeomorphic;

\item
$F'$ is projectively rigid.
\end{enumerate}
After each surgery, we use the natural identification of the new
surface with $F$ to interpret the subsequent curve $\gamma_i$.
\end{thm}

\begin{proof}
Assume first that $F^2=0$, set $f:=f_{X,F}\neq0$, and choose a
representative of every translation class in the finite rim Newton
profile.  Denote the resulting collection of polytopes, with
multiplicities, by $\mathcal P$.  Choose $v_i,d_i$ and $Z$ as in
Lemma~\ref{lem:rigid-zonotope}.  Every primitive class in
$H_1(F;\Z)$ is represented by a nonseparating simple closed curve, so
choose $\gamma_i$ representing $v_i$.

For a positive integer $m$, let
\[
 K_i=T(2,2md_i+1).
\]
Its symmetrized Alexander polynomial is
\[
 \Delta_{K_i}(t)
 =
 \sum_{j=-md_i}^{md_i}
 (-1)^{md_i-j}t^j.
\]
Consequently,
\[
 \Delta_{K_i}(r_{v_i}^2)
 =
 \sum_{j=-md_i}^{md_i}
 (-1)^{md_i-j}e^{2jv_i},
\]
and hence
\[
 {\rm Newt}
 \bigl(\Delta_{K_i}(r_{v_i}^2)\bigr)
 =
 [-2md_iv_i,2md_iv_i].
\]

Set
\[
 q_m
 :=
 \prod_{i=1}^{2g}
 \Delta_{K_i}(r_{v_i}^2).
\]
Every exponent occurring in
$\Delta_{K_i}(r_{v_i}^2)$ has the form $2jv_i\in2R$, so
$q_m\in\Z[2R]$.
By the product rule for Newton polytopes,
\[
\begin{aligned}
 {\rm Newt}(q_m)
 &=
 \sum_{i=1}^{2g}
 {\rm Newt}\bigl(\Delta_{K_i}(r_{v_i}^2)\bigr)\\
 &=
 \sum_{i=1}^{2g}
 [-2md_iv_i,2md_iv_i]\\
 &=
 2m\sum_{i=1}^{2g}[-d_iv_i,d_iv_i]
 =
 2mZ,
\end{aligned}
\]
where the sums are Minkowski sums.  The iterated rim-surgery formula
gives
\[
 f_{X,F'}
 \doteq
 f\prod_{i=1}^{2g}\Delta_{K_i}(r_{v_i}^2)
 =
 fq_m.
\]

Since all the exponents of $q_m$ lie in $2R$, multiplication by
$q_m$ preserves each orbit in $\mathcal A_F/(2R)$.  More precisely,
on the orbit $\lambda$, the polynomial $f_\lambda$ is replaced by
$f_\lambda q_m$.  Thus, if
$[{\rm Newt}(f_\lambda)]$ is represented by
$P\in\mathcal P$, Lemma~\ref{lem:rim-Newton-profile} gives
\[
 \bigl[{\rm Newt}(f_\lambda q_m)\bigr]
 =
 [P+{\rm Newt}(q_m)]
 =
 [P+2mZ].
\]
Consequently,
\[
 \mathcal N_R(f_{X,F'})
 =
 \left\{
 [P+2mZ]\mid P\in\mathcal P
 \right\},
 \tag{2.5}
\]
with multiplicities retained.

Let $\phi\in\mathcal E^\infty(X,F')$, and let
$\Phi\in\rm{Diff}^+(X)$ be an extension.  Under the natural
identification
\[
 R=\mathcal R_{F'}\cong H_1(F;\Z),
\]
the induced action
\[
 A=\phi_*\in\rm{Sp}\bigl(H_1(F;\Z)\bigr)
\]
is also the action induced by $\Phi$ on the rim-torus lattice.
Lemma~\ref{lem:rim-Newton-profile} therefore shows that $A$ preserves
the collection in Equation~(2.5).  Applying
Lemma~\ref{lem:large-zonotope} with scale $2m$, for sufficiently large
$m$, gives $A(Z)=Z$.  Since $A$ is symplectic,
Lemma~\ref{lem:rigid-zonotope} then gives $A=\pm I$.  Therefore
\[
 \rho\bigl(\mathcal E^\infty(X,F')\bigr)=\{1\}.
\]

To distinguish $(X,F')$ from $(X,F)$, we count lattice points in
their rim Newton profiles.  Since the translation ambiguities in the
profile are by elements of $2R\subset R$, the quantity
\[
 \#_R(P):=\lvert P\cap R\rvert
\]
is well defined on the translation classes occurring here and is
preserved by integral affine automorphisms. For every $P\in\mathcal P$, choose a lattice vertex $p\in P\cap R$.
Since $v_1,\ldots,v_{2g}$ form an integral basis,
\[
 \#_R(P+2mZ)
 \geq
 \#_R(p+2mZ)
 =
 \#_R(2mZ)
 =
 \prod_{i=1}^{2g}(4md_i+1),
\]
which tends to infinity with $m$.
Since $\mathcal P$ is finite, after increasing $m$ every polytope in
the rim Newton profile of $F'$ has more lattice points than every
polytope in the profile of $F$.  The two profiles therefore cannot be
equivalent, and naturality of the relative invariant rules out a
diffeomorphism of pairs.

Finally, Lemma~\ref{lem:rim-topology} shows that all the $F_i$ are
topologically isotopic to $F$, while
Lemma~\ref{lem:extendable-conjugacy} identifies their topological
extendable mapping class groups.

Now suppose $F^2=n>0$.  Choose $n$ points of $F$, disjoint from the
rim-surgery regions, and use the corresponding points on all the
surfaces obtained by surgery.  Blow them up and apply the square-zero
case proved above to the proper transform $\widetilde{F}$ in $\widetilde{X}=X\#n\CPb$.
Suppose first that a self-diffeomorphism of the final pair is given
downstairs.  Its restriction to the surface carries the chosen
blow-up points to another $n$-tuple of points.  An isotopy of the
surface carries this tuple back to the chosen one, and isotopy
extension gives an ambient isotopy of the pair with the same property.
After composing with this isotopy, the diffeomorphism lifts to the
proper-transform pair.  The added isotopy restricts to a map isotopic
to the identity on the surface, so it does not change the induced
action on $H_1(F;\Z)$.  Projective rigidity for the square zero surface $\widetilde{F}$ in $\widetilde{X}$
therefore implies projective rigidity of $F$ in $X$.

Similarly, a diffeomorphism between the initial and final pairs
downstairs could be adjusted to match the chosen blow-up points and
then lifted to a diffeomorphism between their proper-transform pairs,
contradicting the square-zero case.  Thus the pairs downstairs are
smoothly inequivalent.  The topological isotopy follows as before from
Lemma~\ref{lem:rim-topology}.
\end{proof}

\medskip
\section{Breaking symmetries by further knotting}
\label{sec:filtration}

We now retain the intermediate stages of the construction in the
previous section.  The initial surfaces will be regular fibers of
Lefschetz fibrations with full monodromy. The reader may turn to \cite{GompfStipsicz1999, BaykurHamada:LF} for more on Lefschetz fibrations.

\subsection{Flexible fibers}

Let
\[
 \pi\colon X\rightarrow S^2
\]
be a Lefschetz fibration with regular fiber
$F=\pi^{-1}(b)$ and critical value set $C$.  Its monodromy
representation is
\[
 \mu_\pi\colon
 \pi_1(S^2\setminus C,b)\rightarrow\M(F).
\]
We say that $\pi$ has \emph{full monodromy} if $\mu_\pi$ is
surjective.  For a Lefschetz pencil, we use the same terminology for
the image in $\M(F)$ obtained after capping the boundary components around the base points.

\begin{prop}
\label{prop:full-monodromy}
Let $F$ be a regular fiber of a Lefschetz fibration or pencil on $X$.
Then
\[
 {\rm im}\,\mu_\pi\subseteq\mathcal E^\infty(X,F).
\]
Thus a regular fiber of a full-monodromy Lefschetz fibration or pencil
is smoothly flexible.  Moreover, in the full-monodromy case,
\[
 \pi_1(X\setminus\nu F)=1
\]
for a Lefschetz fibration.  For a Lefschetz pencil, after blowing up the
base points, the proper transform $\widetilde F\subset\widetilde X$ is
smoothly flexible and
\[
 \pi_1(\widetilde X\setminus\nu\widetilde F)=1.
\]
\end{prop}

\begin{proof}
Let $\gamma\colon[0,1]\to S^2\setminus C$ be a loop based at $b$.
Parallel transport of the fiber along $\gamma$ gives an isotopy of embeddings
\[
 i_t\colon F\hookrightarrow X,
 \qquad
 i_t(F)=\pi^{-1}(\gamma(t)).
\]
At the endpoint, $i_1(F)=F$, and $i_1$ differs from $i_0$ by the
monodromy $\mu_\pi(\gamma)$.  By isotopy extension, $i_t$ extends to an
ambient isotopy $\Phi_t$ of $X$, so $\Phi_1$ preserves $F$ and induces
$\mu_\pi(\gamma)$ on it, giving the desired map.

For a Lefschetz pencil, perform parallel transport away from small
balls about the base points.  In the standard local pencil model, the
isotopy extends across these balls.  The same argument also applies to
the proper transforms after blowing up the base points.

For the fundamental group assertion, apply the following argument to
the given Lefschetz fibration, or to the fibration obtained by blowing
up a pencil.  Let $F$ be its regular fiber and let
$N\triangleleft\pi_1(F)$ be the normal subgroup generated by the
vanishing cycles.  The Lefschetz handle description gives
\[
 \pi_1(X\setminus\nu F)\cong\pi_1(F)/N.
\]
Every Dehn twist along a vanishing cycle acts trivially on this quotient, so $N$ is
invariant under the monodromy group.  At least one vanishing cycle is
nonseparating, since otherwise the monodromy would act trivially on
$H_1(F;\Z)$.  If the monodromy is full, $N$ therefore contains every
nonseparating simple closed curve, and hence a standard generating set
for $\pi_1(F)$.  Thus $N=\pi_1(F)$.
\end{proof}

This is the surface-pushing argument used explicitly by Banerjee and Salter \cite[Corollary~1.12]{BanerjeeSalter2025},  see also \cite{Knavel} for a similar $\pi_1$ argument. Pancholi and Presas state a symplectic refinement for the Dehn twist
about a vanishing cycle
\cite[Lemma~4.3]{PancholiPresas2021}; only the smooth statement above
is used here.

We next record that full-monodromy fibrations with the other properties
needed in our construction exist in every genus.

\begin{lem}
\label{lem:flexible-fibers}
For every $g\geq1$, there is a closed simply connected symplectic
$4$--manifold $X_g$ admitting a genus-$g$ Lefschetz fibration
$X_g\to S^2$ with a section and full monodromy.  Its regular fiber
$F_0$ is smoothly flexible and satisfies
\[
 F_0^2=0,\qquad
 \pi_1(X_g\setminus\nu F_0)=1,
 \qquad
 f_{X_g,F_0}\neq0.
\]
\end{lem}

\begin{proof}
For $g=1$, take the standard elliptic fibration on $E(2)$, whose
monodromy factorization is
\[
 (t_at_b)^{12}=1.
\]
The twists $t_a,t_b$ generate $\M(T^2)$, so the fibration has full
monodromy.

Now let $g\geq2$.  Let
\[
 Y_g=\CP\#(4g+5)\CPb
\]
and consider the hyperelliptic genus-$g$ Lefschetz fibration with
monodromy factorization
\[
 W_g:=h_g^2=1,
 \qquad
 h_g=
 t_{c_1}\cdots t_{c_{2g}}
 t_{c_{2g+1}}^2
 t_{c_{2g}}\cdots t_{c_1}.
\]
Let $b$ be the extra nonseparating curve for which
$c_1,\ldots,c_{2g},b$ form a Humphries generating system
\cite{humphries:generators}, and choose
$\psi\in\M(\Sigma_g)$ with $\psi(c_1)=b$.  Form the twisted fiber sum
of two copies of this fibration using $\psi$, as in
\cite[Section~3]{baykur-korkmaz-simone:geography}.  Its monodromy
factorization is
\[
 W_g\bigl(\psi W_g\psi^{-1}\bigr)=1.
 \tag{3.1}
\]
The first factor contains $t_{c_1},\ldots,t_{c_{2g}}$, while the
second contains
\[
 \psi t_{c_1}\psi^{-1}=t_\psi(c_1)=t_b.
\]
Hence the resulting fibration has full monodromy.

Choose a representative of $\psi$ carrying the intersection point of
one section with the fiber to that of the other; the punctured sections
then glue to a section.  By
Proposition~\ref{prop:full-monodromy}, its regular fiber $F_0$ is
smoothly flexible and
\[
 \pi_1(X_g\setminus\nu F_0)=1.
\]
In particular, $X_g$ is simply connected.  The fiber sum is
symplectic, and $F_0$ is a symplectic, primitively embedded surface.
Thus $(X_g,F_0)$ has nonzero relative Seiberg--Witten invariant
\cite[Theorem~1.1]{FintushelStern1997}
\cite{FintushelSternAddendum}.
\end{proof}

\begin{remark}[Pencil examples]
There are also many examples of Lefschetz pencils with full monodromy.

\smallskip
\noindent\textup{(1)}
For every $g\geq2$, one can derive a full-monodromy Lefschetz pencil
with one base point from the $2g$--chain relation.  The left-hand side
of this relation is Hurwitz equivalent to
\[
 (t_{c_1}\cdots t_{c_{2g}})^{4g+2}
 \sim
 (t_{c_1}t_{c_2}t_{c_3})^4
 (t_{c_1}\cdots t_{c_{2g}})^{2g+1}U_g
\]
for some positive word $U_g$.  Applying the $3$--chain relation gives
\[
 t_\delta
 =
 t_{d_1}t_{d_2}
 (t_{c_1}\cdots t_{c_{2g}})^{2g+1}U_g,
\]
where, in the standard configuration, one of $d_1,d_2$ is the extra
curve in a Humphries generating system
\cite{humphries:generators}.  The capped monodromy is therefore all of
$\M(\Sigma_g)$.

This positive factorization of $t_\delta$ determines a Lefschetz pencil
with one base point.  Its regular fiber $F$ has $F^2=1$ and is smoothly
flexible.  Moreover, $X\setminus\nu F$ is a Lefschetz fibration over a
disk with fiber $\Sigma_g^1$, and the chain vanishing cycles normally
generate $\pi_1(\Sigma_g^1)$.  Hence
\[
 \pi_1(X\setminus\nu F)=1.
\]
The fiber is symplectic and primitive, so $(X,F)$ has nonzero relative
Seiberg--Witten invariant.

\smallskip
\noindent\textup{(2)}
Banerjee and Salter provide many complex-algebraic examples.  They show
that sufficiently ample pencils on simply connected projective
surfaces have monodromy equal to the $r$--spin mapping class group
determined by a maximal root of the adjoint bundle
\cite[Corollary~1.8]{BanerjeeSalter2025}.  When $r=1$, the monodromy is
full and the fiber is smoothly flexible
\cite[Corollary~1.12]{BanerjeeSalter2025}.  After blowing up the base
points, Proposition~\ref{prop:full-monodromy} gives a square-zero
flexible fiber with simply connected complement.  These yield many
further complex initial pairs, though not one in every prescribed
genus.
\end{remark}

\subsection{Successive symmetry bounds}

For a relative invariant $f_{X,F}$, let
$\mathcal N_R(f_{X,F})$ denote its rim Newton profile, as in
Section~\ref{sec:rigid}.  Flexibility has a particularly strong
consequence for the initial profile.

\begin{prop}[Rim Newton obstruction to smooth flexibility]
\label{prop:flexible-profile}
Let $F\subset X$ be smoothly flexible and suppose that
$f_{X,F}\neq0$.  Every polytope in the rim Newton profile
$\mathcal N_R(f_{X,F})$ is a point.
\end{prop}

\begin{proof}
Set
\[
R:=\mathcal R_F\cong H_1(F;\Z).
\]
For the polynomials $f_\lambda$ used in the definition of the rim
Newton profile, set
\[
\mathcal D
=
\bigcup_{\lambda}
\left\{
u-u'
\ \middle|\
u,u'\in{\rm supp}(f_\lambda)
\right\}
\subset 2R,
\]
where $\lambda$ ranges over the $2R$--orbits for which
$f_\lambda\neq0$. This is a finite set, and it is independent of the
chosen origins $\alpha_\lambda$, since translating a support does not
change its pairwise differences.

Naturality of the relative invariant shows that $\mathcal D$ is
invariant under the action on $R$ of
$\mathcal E^\infty(X,F)$. Since $F$ is smoothly flexible and
\[
\M(F)\rightarrow
\rm{Sp}\bigl(H_1(F;\Z)\bigr)
\]
is surjective, $\mathcal D$ is invariant under the full symplectic
group.

Every nonzero element $u\in H_1(F;\Z)$ has an infinite symplectic
orbit. Indeed, choose $w$ with $u\cdot w\neq0$. The symplectic
transvection
\[
T_w(x)=x+(x\cdot w)w
\]
satisfies
\[
T_w^k(u)=u+k(u\cdot w)w,
\qquad k\in\Z,
\]
and these elements are pairwise distinct. Since $\mathcal D$ is
finite, it follows that $\mathcal D=\{0\}$.

Thus each ${\rm supp}(f_\lambda)$ consists of a single exponent, so
each ${\rm Newt}(f_\lambda)$, and hence every polytope in
$\mathcal N_R(f_{X,F})$, is a point.
\end{proof}

\begin{remark}
\label{rmk:flexible-profile-general}
The proof uses smooth flexibility only through the size of the
homological image
\[
\Gamma_{X,F}:=
{\rm im}\left(
\mathcal E^\infty(X,F)
\rightarrow
\rm{Sp}\bigl(H_1(F;\Z)\bigr)
\right).
\]
The same conclusion holds whenever every nonzero class in
$H_1(F;\Z)$ has an infinite $\Gamma_{X,F}$ orbit; in particular, it
holds when $\Gamma_{X,F}$ has finite index in the full symplectic
group. Thus the presence of a positive-dimensional polytope in the
rim Newton profile obstructs smooth flexibility.
\end{remark}

Let $v_1,\ldots,v_{2g}$ and $d_1,\ldots,d_{2g}$ be the basis and
pairwise distinct positive integers from
Lemma~\ref{lem:rigid-zonotope}.  Put $P:=\rm{PSp}(2g,\Z)$ and, for $j=1,\ldots,2g$, let
$\ell_j:=\Z v_j$.  Set
\[
 Z_i
 =
 \sum_{j=1}^{i}[-d_jv_j,d_jv_j],
 \qquad
 P_i
 =
 \bigcap_{j=1}^{i}\rm{Stab}_{P}(\ell_j),
\]
with $P_0=P$.

\begin{lem}
\label{lem:projective-filtration}
The stabilizer of $Z_i$ in $P$ is $P_i$, and
\[
 P
 =
 P_0
 \supsetneq
 P_1
 \supsetneq
 \cdots
 \supsetneq
 P_{2g}
 =
 \{1\}.
\]
\end{lem}

\begin{proof}
Since the $v_j$ are linearly independent, $Z_i$ is an
$i$--dimensional parallelotope.  Its edge directions are
$\ell_j\otimes\R=\R v_j$, and their lattice lengths are the pairwise
distinct numbers $2d_j$.  Thus an integral automorphism preserving
$Z_i$ must preserve every $\ell_j$, for $j\leq i$.  Conversely, an
integral automorphism preserving these lines sends each primitive
vector $v_j$ to $\pm v_j$, and hence preserves $Z_i$.  Therefore the
stabilizer of $Z_i$ in $P$ is $P_i$.

To see that each inclusion is strict, set
\[
 V=H_1(\Sigma_g;\Z),
 \qquad
 U_{i-1}=\langle v_1,\ldots,v_{i-1}\rangle_{\Z}.
\]
Since $v_1,\ldots,v_{2g}$ form an integral basis, there is an integral
linear functional $\lambda_i\colon V\to\Z$ satisfying
\[
 \lambda_i(v_j)=\delta_{ij}.
\]
The intersection form on $V$ is unimodular, so the map
\[
 V\rightarrow{\rm Hom}(V,\Z),
 \qquad
 w\longmapsto (x\mapsto x\cdot w),
\]
is an isomorphism.  Hence there is an integral class $w_i\in V$ such
that
\[
 v_j\cdot w_i=0\quad(j<i),
 \qquad
 v_i\cdot w_i=1.
\]
In particular, $w_i\in U_{i-1}^{\perp}$ but is not proportional to
$v_i$.  The symplectic transvection
\[
 T_{w_i}(x)=x+(x\cdot w_i)w_i
\]
fixes $\ell_1,\ldots,\ell_{i-1}$, whereas
\[
 T_{w_i}(v_i)=v_i+w_i\notin\ell_i.
\]
Thus the projective class of $T_{w_i}$ lies in
$P_{i-1}\setminus P_i$, proving
\[
 P_{i-1}\supsetneq P_i.
\]

Finally, $Z_{2g}=Z$, and
Lemma~\ref{lem:rigid-zonotope} says that its stabilizer in
$\rm{Sp}(2g,\Z)$ is $\{\pm I\}$.  Its stabilizer in
$P=\rm{PSp}(2g,\Z)$ is therefore trivial, so
$P_{2g}=\{1\}$.
\end{proof}

The preceding lemma gives the desired projective symmetry bounds.  We
now realize them by successive rim surgeries.

\begin{thm}
\label{thm:filtration-detailed}
Let $(X_g,F_0)$ be the pair in
Lemma~\ref{lem:flexible-fibers}, and let
$v_i,d_i$ and $P_i$ be as above.  There are oriented simple closed
curves $\gamma_i\subset F_0$ representing $v_i$ such that, using the
natural identifications of the abstract surfaces after each surgery
and setting
\[
 K_i=T(2,2d_i+1),
 \qquad
 F_i=(F_{i-1})_{K_i,\gamma_i},
 \qquad i=1,\ldots,2g,
\]
where $K_i$ is the $(2,2d_i+1)$--torus knot, the surfaces $F_i$ are
mutually topologically isotopic and topologically flexible, the pairs
$(X_g,F_i)$ are pairwise nondiffeomorphic, and
\[
 \rho\bigl(\mathcal E^\infty(X_g,F_i)\bigr)\subseteq P_i
\]
for every $i=0,\ldots,2g$.
\end{thm}

\begin{proof}
Set $R=\mathcal R_{F_0}$.  By
Proposition~\ref{prop:flexible-profile}, every polytope in
$\mathcal N_R(f_{X_g,F_0})$ is a point.  For the torus knot
\[
 K_i=T(2,2d_i+1),
\]
the symmetric Alexander polynomial has extremal exponents $\pm d_i$,
and hence
\[
 {\rm Newt}\bigl(\Delta_{K_i}(r_{v_i}^2)\bigr)
 =
 [-2d_iv_i,2d_iv_i].
\]
It follows from the product rule that
\[
\begin{aligned}
 {\rm Newt}\left(
 \prod_{j=1}^i\Delta_{K_j}(r_{v_j}^2)
 \right)
 &=
 \sum_{j=1}^i[-2d_jv_j,2d_jv_j]\\
 &=
 2Z_i.
\end{aligned}
\]
The iterated rim-surgery formula and
Lemma~\ref{lem:rim-Newton-profile} therefore show that every polytope
in $\mathcal N_R(f_{X_g,F_i})$ is translation equivalent to $2Z_i$.
Equivalently,
\[
 [P]=[2Z_i]
 \qquad
 \text{for every }[P]\in\mathcal N_R(f_{X_g,F_i}).
 \tag{3.3}
\]
The number of occurrences of $[2Z_i]$ is the number of orbits
$\lambda\in\mathcal A_{F_0}/(2R)$ for which $f_\lambda\neq0$, and is
independent of $i$.

Let $\phi\in\mathcal E^\infty(X_g,F_i)$.  Naturality of the relative
invariant implies that $\phi_*$ preserves the translation class of
$2Z_i$.  Thus $\phi_*(2Z_i)$ is a translate of $2Z_i$.  Since both
$\phi_*(2Z_i)$ and $2Z_i$ are centrally symmetric about the origin,
this translation is zero, and hence
$
 \phi_*(Z_i)=Z_i.
$
Lemma~\ref{lem:projective-filtration} now gives
$
 \rho(\phi)\in P_i.
$
This proves the asserted symmetry bounds.

By construction $Z_i$
is generated by the $i$ linearly independent directions
$v_1,\ldots,v_i$, and hence has affine dimension $i$.  Thus every
polytope in Equation~(3.3), being translation-equivalent to $2Z_i$,
also has affine dimension $i$.  Diffeomorphisms of pairs
preserve the rim Newton profile and the affine dimensions of its
polytopes.  The pairs $(X_g,F_i)$ are therefore pairwise
nondiffeomorphic.

Finally, ordinary rim surgery preserves simple connectivity of the
complement.  Lemma~\ref{lem:rim-topology} shows that all the $F_i$ are
topologically isotopic to $F_0$.  Since $F_0$ is smoothly flexible, it
is topologically flexible, and
Lemma~\ref{lem:extendable-conjugacy} shows that every $F_i$ is
topologically flexible.
\end{proof}

Together with Lemma~\ref{lem:projective-filtration}, this proves
Theorem~\ref{thm:b}.

\begin{example}
\label{ex:g1-rim-newton}
We illustrate the polynomials $f_\lambda$, their Newton polytopes,
and the symmetry bounds in genus $1$.  Let $F=F_0\subset E(2)$ be a
regular fiber of the standard elliptic fibration, and choose an
oriented basis
\[
 a,b\in H_1(F;\Z),
\]
so that
\[
 R=\mathcal R_F\cong H_1(F;\Z)=\Z a\oplus\Z b.
\]
Since $F$ is smoothly flexible and $f_{E(2),F}\neq0$,
Proposition~\ref{prop:flexible-profile} shows that every polytope in its rim
Newton profile is a point.  Fix an orbit $\lambda$ with
$f_\lambda\neq0$.  After choosing its affine origin, we may write
\[
 f_\lambda=c_\lambda
\]
for some $c_\lambda\in\Z\setminus\{0\}$.

\begin{figure}[t]
\centering
\begin{tikzpicture}[scale=0.98]
\tikzset{
  support/.style={
    fill=blue!75!black,
    draw=blue!75!black,
    line width=0.5pt
  },
  segment/.style={
    draw=blue!75!black,
    line width=1.45pt
  },
  hull/.style={
    draw=blue!75!black,
    fill=blue!22,
    line width=1.25pt
  },
  axis/.style={<->,thin},
  lab/.style={font=\small},
  title/.style={font=\small},
  stab/.style={font=\small}
}

\begin{scope}[xshift=0cm]
  \draw[axis] (-2.35,0) -- (2.55,0)
    node[right,lab] {$2a$};
  \draw[axis] (0,-2.65) -- (0,2.75)
    node[above,lab] {$2b$};

  \node[title,anchor=west] at (-2.20,2.45) {$F=F_0$};

  \filldraw[support] (0,0) circle (2.4pt);

  \node[stab] at (0,-3.10) {$P_0=\rm{PSp}(2,\Z)$};
\end{scope}

\begin{scope}[xshift=5.70cm]
  \draw[axis] (-2.35,0) -- (2.55,0)
    node[right,lab] {$2a$};
  \draw[axis] (0,-2.65) -- (0,2.75)
    node[above,lab] {$2b$};

  \node[title,anchor=west] at (-2.20,2.45) {$F_1$};

  \draw[segment] (-1,0) -- (1,0);
  \filldraw[support] (-1,0) circle (2.2pt);
  \filldraw[support] (0,0) circle (2.2pt);
  \filldraw[support] (1,0) circle (2.2pt);

  \node[stab] at (0,-3.10)
    {$P_1=\rm{Stab}_P(\Z a)$};
\end{scope}

\begin{scope}[xshift=11.40cm]
  \filldraw[hull] (-1,-2) rectangle (1,2);

  \draw[axis] (-2.35,0) -- (2.55,0)
    node[right,lab] {$2a$};
  \draw[axis] (0,-2.65) -- (0,2.75)
    node[above,lab] {$2b$};

  \node[title,anchor=west] at (-2.20,2.45) {$F'=F_2$};

  \foreach \x in {-1,0,1}
    \foreach \y in {-2,-1,0,1,2}
      \filldraw[support] (\x,\y) circle (1.95pt);

  \node[stab] at (0,-3.10) {$P_2=\{1\}$};
\end{scope}

\end{tikzpicture}

\caption{Supports (blue points) and their Newton polytopes for the orbit polynomials associated to $(E(2),F_i)$, $i=0,1,2$.  The labels below record their stabilizers in $\rm{PSp}(2,\Z)$.}
\label{fig:g1-rim-newton}
\end{figure}
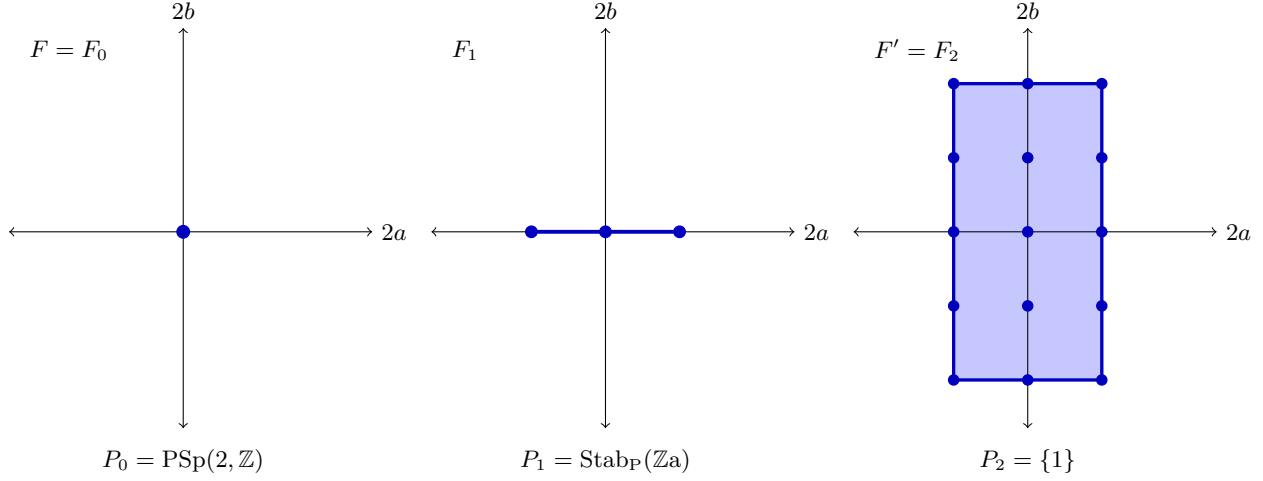

Perform rim surgery first along a curve representing $a$ using
$T(2,3)$, and then along a curve representing $b$ using $T(2,5)$.
Denote the resulting surfaces by $F_1$ and $F'=F_2$.  Since
\[
 \Delta_{T(2,3)}(t)=t^{-1}-1+t,
 \qquad
 \Delta_{T(2,5)}(t)=t^{-2}-t^{-1}+1-t+t^2,
\]
the corresponding orbit polynomials satisfy
\[
 f^{(1)}_\lambda
 \doteq
 c_\lambda(e^{-2a}-1+e^{2a})
\]
and
\[
 f^{(2)}_\lambda
 \doteq
 c_\lambda
 (e^{-2a}-1+e^{2a})
 (e^{-4b}-e^{-2b}+1-e^{2b}+e^{4b}).
\]
Consequently,
\[
 {\rm Newt}(f_\lambda)=\{0\},
 \qquad
 {\rm Newt}(f^{(1)}_\lambda)=[-2a,2a],
\]
and
\[
 {\rm Newt}(f^{(2)}_\lambda)
 =
 [-2a,2a]+[-4b,4b]
 =
 2Z,
 \qquad
 Z=[-a,a]+[-2b,2b].
\]

Set $P=\rm{PSp}(2,\Z)$.  The point $\{0\}$ is preserved by all of
$P$, whereas the segment $[-2a,2a]$ has stabilizer
\[
 P_1=\rm{Stab}_P(\Z a).
\]
For the final rectangle, the unequal weights prevent its two edge
directions from being exchanged.  Any symplectic lattice
automorphism preserving it must therefore preserve both $\Z a$ and
$\Z b$, and hence is $\pm I$ before passing to $P$.  Thus
\[
 P=P_0
 \supsetneq
 P_1=\rm{Stab}_P(\Z a)
 \supsetneq
 P_2=\{1\}.
\]
Since the three polytopes are centrally symmetric about the origin,
the same statements hold for their translation classes.  This is the
genus $1$ model for the symmetry bounds in
Theorem~\ref{thm:filtration-detailed}.

\end{example}

\begin{remark} \label{rk:top-flexible}
In a recent preprint, 
Pyronneau obtains an extension theorem for surfaces
with boundary properly embedded in simply connected $4$--manifolds with boundary $S^3$ \cite[Theorem~4.3]{Pyronneau2026}. Puncturing a closed pair in a
standard ball pair gives the corresponding closed statement: every
closed ordinary oriented surface with simply connected complement is
topologically flexible.
\end{remark}

\begin{remark}
\label{rmk:plane-curves}
Let $C_d\subset\CP$ be a nonsingular plane curve of degree $d$, and set
\[
 g_d=\frac{(d-1)(d-2)}2.
\]
Hirose proved that $C_3$ and $C_4$ are smoothly flexible
\cite[Theorem~4.2]{Hirose2005}.  In \cite{LehmanLewis}, Lehman and
Lewis show that, for $d\geq5$,
\[
 \mathcal E^0(\CP,C_d)
 =
 \begin{cases}
 \M(C_d), & d\text{ even},\\
 \M(C_d,q_d), & d\text{ odd},
 \end{cases}
\]
where $q_d$ is the Rokhlin quadratic form and $\M(C_d,q_d)$ is the
corresponding spin mapping class group.

Here
\[
 \pi_1(\CP\setminus\nu C_d)\cong\Z_d.
\]
The smooth extendable group contains the plane-curve monodromy group,
which has finite index in $\M(C_d)$ for $d\geq5$
\cite{Salter2019,Salter2025}.  Hence Remark~\ref{rmk:flexible-profile-general} shows that every
polytope in the initial rim Newton profile is a point. Using the $1$--twist version of
Remark~\ref{rmk:cyclic-rim}, the proof of
Theorem~\ref{thm:filtration-detailed} gives, for every $d\geq3$, a
sequence
\[
 C_{d,0}=C_d,C_{d,1},\ldots,C_{d,2g_d}\subset\CP
\]
of topologically isotopic, pairwise nondiffeomorphic surfaces with the
same successive projective symmetry bounds. They are all topologically flexible when $d=3$ or $d$ is even, and their topological extendable group is the finite index subgroup $\M(C_d,q_d)$ when $d\geq5$ is odd.
\end{remark}

\begin{remark}
Rigidity is highly unstable under internal stabilization.
Let $T\subset S^4$ be the standard unknotted torus and let 
\[
 \widehat F_i:=F_i\#T \subset X_g \# S^4 \cong X_g
\]
be a standard internal stabilization.  The argument of
\cite[Theorem~2 and Section~3.1]{BaykurSunukjian2016} may be repeated
while reusing the same trivial handle, so all the $\widehat F_i$ are
smoothly isotopic.

Every mapping class of $F_0$ has a representative fixing the
stabilization disk pointwise.  By smooth flexibility and isotopy
extension, its ambient extension may be chosen to be the identity on
the stabilization ball, and hence extends over $\widehat F_0$ by the
identity on the added handle.  Thus the homological image of
$\mathcal E^\infty(X_g,\widehat F_0)$, and hence that of each
$\mathcal E^\infty(X_g,\widehat F_i)$ after conjugation, contains a
copy of $\rm{Sp}(2g,\Z)$.  In particular, none of the once-stabilized
surfaces is projectively rigid.
\end{remark}

\medskip
\section{A topologically and smoothly rigid surface}
\label{sec:hyperbolic}

We now turn to a complementary source of rigidity from hyperbolic
geometry.  The proof has two parts.  First, Mostow rigidity and the
Dehn--Nielsen--Baer theorem identify the extendable mapping classes of
a totally geodesic surface with restrictions of ambient isometries.
We then use arithmetic embedding results to place a closed hyperbolic
surface with trivial orientation-preserving isometry group inside a
closed hyperbolic $4$--manifold.

The geometric facts used in the first part are standard consequences
of Mostow rigidity, the Dehn--Nielsen--Baer theorem, and the basic
theory of cocompact Fuchsian groups; see
\cite{MostowStrongRigidity,FarbMargalit2012,Katok1992}.  We record them
only in the form needed here.

\begin{proposition}
\label{prop:mostow-reduction}
Let $F\subset X$ be a closed embedded totally geodesic surface in a
closed hyperbolic $4$--manifold, and set
\[
 G_F
 =
 \left\{
 u\in\rm{Isom}^+(X)
 \ \middle|\
 u(F)=F\text{ and }u|_F\text{ preserves orientation}
 \right\}.
\]
Then restriction gives
\[
 \mathcal E^\infty(X,F)
 =
 \mathcal E^0(X,F)
 =
 \rm{im}\bigl(G_F\rightarrow\M(F)\bigr).
\]
In particular, if $\rm{Isom}^+(F)=\{1\}$ for the hyperbolic
metric induced on $F$, then
\[
 \mathcal E^\infty(X,F)=\mathcal E^0(X,F)=\{1\}.
\]
\end{proposition}

\begin{proof}
Write $X=\mathbb H^4/\Gamma$.  Choose a component
$P\cong\mathbb H^2$ of the inverse image of $F$, and set 
\[
 H:=\rm{Stab}_\Gamma(P).
\]
Then $F=P/H$.  Since $H$ is a cocompact Fuchsian group, its limit set
is $\partial_\infty P$; see \cite{Katok1992}.  The plane $P$ is the
hyperbolic convex hull of this circle.  It follows that
\begin{equation}
\label{eq:surface-self-normalizing}
 N_\Gamma(H)=H.
\end{equation}
Indeed, an element normalizing $H$ preserves its limit set, hence
preserves $P$, and therefore belongs to
$\rm{Stab}_\Gamma(P)=H$.

Let $[\varphi]\in\mathcal E^0(X,F)$ be represented by an
orientation-preserving homeomorphism
\[
 \Phi\colon(X,F)\rightarrow(X,F).
\]
A lift of $\Phi$ sends $P$ to another component of the inverse image
of $F$.  After composing the lift with a deck transformation, we may
assume that it preserves $P$.  Let
$\alpha\in\rm{Aut}(\Gamma)$ be the induced automorphism.
Then $\alpha(H)=H$.

By Mostow rigidity \cite{MostowStrongRigidity}, there is an
orientation-preserving isometry $u$ of $X$ inducing the same outer
automorphism of $\Gamma$ as $\Phi$.  After changing a lift of $u$ by a
deck transformation, we may assume that its induced automorphism
$\beta\in\rm{Aut}(\Gamma)$ also preserves $H$.  The lift
then preserves the limit set of $H$, and hence $P$, so $u(F)=F$.

The automorphisms $\alpha$ and $\beta$ determine the same outer
automorphism of $\Gamma$, so they differ by conjugation by some
$\gamma\in\Gamma$.  Since both preserve $H$, we have
$\gamma\in N_\Gamma(H)=H$.  Thus their restrictions to $H$ determine
the same element of $\rm{Out}^+(H)$.  Under the
identification $H\cong\pi_1(F)$, the Dehn--Nielsen--Baer theorem
\cite[Theorem~8.1]{FarbMargalit2012} gives
\[
 [\Phi|_F]=[u|_F]\in\M(F).
\]
In particular, $u\in G_F$, and hence
\[
 \mathcal E^0(X,F)
 \subseteq
 \rm{im}\bigl(G_F\rightarrow\M(F)\bigr).
\]
The reverse inclusions
\[
 \rm{im}\bigl(G_F\rightarrow\M(F)\bigr)
 \subseteq
 \mathcal E^\infty(X,F)
 \subseteq
 \mathcal E^0(X,F)
\]
are immediate.
\end{proof}

We next record the arithmetic input we will use to produce an ambient $4$--manifold.

\begin{lem}
\label{lem:two-step-arithmetic-embedding}
Let $F$ be a closed oriented arithmetic hyperbolic surface defined
over a number field $k\ne\Q$.  Then $F$ admits a totally geodesic
embedding in a closed oriented hyperbolic $4$--manifold.
\end{lem}

\begin{proof}
Since $F$ has even dimension, its arithmetic lattice is of simplest
type and, after conjugation, lies in
\[
 \rm{SO}^+(q,k)\subset\rm{O}(q,k)
\]
for an admissible ternary quadratic form $q$ over $k$; see
\cite[Corollary~1.2 and Section~4.2]{KRSArithmeticEmbedding}.
Applying \cite[Lemma~5.1]{MartelliRioloSlavichSpin} gives a totally
geodesic embedding
\[
 F\lhook\joinrel\rightarrow Y^3,
\]
where $Y^3$ is an oriented arithmetic hyperbolic $3$--manifold defined
over the same field $k$, and its fundamental group again lies in the
$k$--points of the corresponding orthogonal group.  We may therefore
apply the same lemma once more to obtain a totally geodesic embedding
\[
 Y^3\lhook\joinrel\rightarrow X^4.
\]
The compactness clause in that lemma shows successively that $Y^3$ and
$X^4$ are compact, because the input is compact and $k\ne\Q$ at
each stage.  Moreover, the lemma embeds the input manifold itself,
rather than only a finite cover.  Thus the composite embeds the
original surface $F$ in a closed oriented hyperbolic $4$--manifold.
\end{proof}

\begin{proof}[Proof of Theorem~\ref{thm:c}]
Maclachlan constructs a torsion-free maximal arithmetic Fuchsian group
$\Gamma_0<\rm{PSL}(2, \R)$, defined over a totally
real cubic field, such that
\[
 F=\mathbb H^2/\Gamma_0
\]
is a closed surface of genus $40$
\cite[Theorem~3.1, Section~6, and
Example~6.1]{MaclachlanTorsion}. 

We claim that $\rm{Isom}^+(F)=\{1\}$.  Lifting isometries to
the universal cover gives
\[
 \rm{Isom}^+(F)
 \cong
 N_{\rm{PSL}(2, \R)}(\Gamma_0)/\Gamma_0;
\]
see \cite{Katok1992}.  Since $F$ is closed, this group is finite.
Hence the normalizer is a finite-index Fuchsian overgroup of
$\Gamma_0$, and therefore an arithmetic Fuchsian group in the same
commensurability class; see
\cite[Chapter~8]{MaclachlanReid2003}.  Maximality of $\Gamma_0$
forces
\[
 N_{\rm{PSL}(2, \R)}(\Gamma_0)=\Gamma_0,
\]
proving the claim.

The defining field of $F$ has degree three, so it is not $\Q$.
Lemma~\ref{lem:two-step-arithmetic-embedding} embeds $F$ totally
geodesically in a closed hyperbolic $4$--manifold $X$.  Proposition
\ref{prop:mostow-reduction} now gives $
 \mathcal E^\infty(X,F)=\mathcal E^0(X,F)=\{1\}.$
\end{proof}

\begin{remark}
\label{rem:hyperbolic-rigid-recipe}
The proof gives a general recipe.  Any closed oriented arithmetic
hyperbolic surface $F$, defined over a number field $k\ne\Q$ and
satisfying $\rm{Isom}^+(F)=\{1\}$, admits a totally geodesic embedding
in a closed hyperbolic $4$--manifold for which
\[
 \mathcal E^\infty(X,F)=\mathcal E^0(X,F)=\{1\}.
\]
The genus-$40$ surface above is one explicit input.  As pointed out to
us by Leone Slavich, varying the coefficients adjoined to the defining
quadratic form in the two extension steps should yield infinitely many
pairwise noncommensurable closed arithmetic hyperbolic $4$--manifolds,
each containing the same surface $F$ isometrically as a totally
geodesic submanifold.  By Mostow rigidity, these manifolds would be
pairwise nonhomeomorphic.

The ambient manifolds are obtained only through existence theorems, so
the construction provides no explicit combinatorial models for them.
Moreover, they are not simply connected and $F$ is topologically rigid
rather than flexible, so these examples do not answer
Question~\ref{que:flexible-rigid}.
\end{remark}

\begin{remark}
Every surface $F'\subset X$ topologically isotopic to the surface in
Theorem~\ref{thm:c} is again rigid in both categories.  Indeed,
Lemma~\ref{lem:extendable-conjugacy} identifies its topological
extendable group with $\mathcal E^0(X,F)=\{1\}$, and $\mathcal E^\infty(X,F')
 \subseteq
 \mathcal E^0(X,F')$ implies the same for the smooth extendable group.
This does not rule out exotic copies; it shows only that any such copy
would also be rigid in both categories.
\end{remark}

\section{Final comments}
\label{sec:questions}

We conclude with three natural questions.

\begin{question}
\label{que:flexible-rigid}
For every $g\geq 1$, does there exist a topologically flexible but
smoothly rigid embedding $\Sigma_g\hookrightarrow X$ in a closed
simply connected smooth $4$--manifold?
\end{question}

The genus-$1$ case of Theorem~\ref{thm:a} comes very close: the only nontrivial mapping class our argument does not exclude is the central element $-I$, represented by the hyperelliptic  involution.  In higher genus, $\ker\rho$ also contains the Torelli group.  The missing sign reflects the reciprocity (often called the symmetry) of the Alexander polynomial in the relative Seiberg--Witten invariants of rim-surgered surfaces.

It is plausible that, if the topological flexibility condition is
dropped, peripheral data associated to a nontrivial
surface-complement group may obstruct the remaining involution;
cf.\ \cite{Niu2026}. 

\begin{question}
\label{que:hyperbolic-exotic}
Let $F\subset X$ be a closed oriented totally geodesic surface in a
closed oriented hyperbolic $4$--manifold, and suppose that $F$ is
rigid in both categories.  Can $F$ admit an exotic copy?  In
particular, must every surface topologically isotopic to the
surface in Theorem~\ref{thm:c} be smoothly isotopic to it?
\end{question}

As Theorem~\ref{thm:b} provides only nested upper bounds
$P_i\supseteq\rho\bigl(\mathcal E^\infty(X_g,F_i)\bigr)$, we also ask:

\begin{question}
\label{que:nested-groups}
Can the surfaces in Theorem~\ref{thm:b} be chosen, with 
identifications $F_i\cong\Sigma_g$, so that we get 
\[
 \M(\Sigma_g)
 =
 \mathcal E^\infty(X_g,F_0)
 \supsetneq
 \mathcal E^\infty(X_g,F_1)
 \supsetneq\cdots\supsetneq
 \mathcal E^\infty(X_g,F_{2g})?
\]
Can such a sequence be infinite?
\end{question}

Strict nesting of the extendable groups would require realization results complementing our obstructions.  Continuing the present construction beyond the projectively rigid stage would require invariants finer than projective homology.

\medskip
\appendix
\section{Embedded surfaces are determined by their ambient symmetries}
\label{appendix:kleinwasright}

The classical reconstruction theorems of Whittaker and Filipkiewicz
show that a closed manifold is determined by its homeomorphism and
diffeomorphism groups in the TOP and DIFF categories
\cite{Whittaker1963,filipkiewicz1982}.  The corresponding relative
statement is a consequence of more general reconstruction results for
groups preserving submanifolds.  We record only the case needed here,
for surfaces in $4$--manifolds, and include the argument for
completeness.

For a closed $4$--manifold $X$ and a nonempty closed surface
$F\subset X$, set
\[
 \rm{Diff}(X,F)
 :=
 \{g\in\rm{Diff}(X)\mid g(F)=F\}
 \text{ and }
 \rm{Homeo}(X,F)
 :=
 \{g\in\rm{Homeo}(X)\mid g(F)=F\}.
\]
Here no orientation condition is imposed, and $F$ need not be orientable or connected.

\begin{theorem}
\label{thm:pair-reconstruction}
Let $F_i\subset X_i$, $i=1,2$, be nonempty closed, possibly
disconnected, surfaces in connected closed $4$--manifolds.  In the smooth category, every isomorphism
\[
 \Psi\colon
 \rm{Diff}(X_1,F_1)
 \xrightarrow{\;\cong\;}
 \rm{Diff}(X_2,F_2)
\]
has the form
$\Psi(g)=h  g  h^{-1}$ for some diffeomorphism of pairs
$h\colon(X_1,F_1)\to(X_2,F_2)$.  In the topological category, provided the $F_i$ are locally flat, the
analogous statement holds with homeomorphisms in place of
diffeomorphisms.
\end{theorem}

\begin{proof}
In the smooth category, Rybicki's refinement of Filipkiewicz's
reconstruction theorem yields the statement asserted in the theorem:  Main Theorem and
Example~7 in \cite{RybickiAdmissible} apply to diffeomorphism groups preserving a positive-dimensional submanifold.

For the topological statement, set
$H_i:=\rm{Homeo}(X_i,F_i)$, and let $H_i^0$ denote the identity
component of $H_i$ in the compact-open topology.  Mann proves that
every homomorphism from $H_i^0$ to a separable topological group is
continuous \cite[Theorem~4.7]{MannAutomatic}.  Since $H_2$ is
separable, the restriction
$\Psi|_{H_1^0}\colon H_1^0\to H_2$ is continuous, so its image lies in
$H_2^0$.  Applying the same argument to $\Psi^{-1}|_{H_2^0}$ gives
\[
 \Psi(H_1^0)=H_2^0.
\]

We now verify the two hypotheses of the reconstruction theorem of
Ben Ami--Rubin \cite[Theorem~I]{BenAmiRubin}.  For the first
hypothesis, let $\mathcal U$ be an open cover of $X_i$, and let
$K\leq H_i^0$ be the subgroup generated by elements supported in
members of $\mathcal U$.  After passing to a finite subcover, Mann's
relative fragmentation result
\cite[Proposition~2.3]{MannAutomatic} shows that the subgroup generated
by elements supported in members of this finite subcover contains an
open neighborhood $V$ of the identity.  This subgroup is contained in
$K$, so $V\subset K$.  Hence $K$ is open, being a union of translates
of $V$.  Every open subgroup of a topological group is also closed.
Since $H_i^0$ is connected, it follows that $K=H_i^0$.

For the second hypothesis, no point of $X_i$ is fixed by every element
of $H_i^0$: points outside $F_i$ can be moved by ball-supported
isotopies, while points of $F_i$ can be moved tangentially in locally
flat charts by compactly supported ambient isotopies preserving $F_i$
setwise.  The two hypotheses of Ben Ami--Rubin's theorem are therefore
satisfied, so it gives a homeomorphism $h\colon X_1\to X_2$ such that
\[
 \Psi(g)=hgh^{-1}
 \qquad\text{for every }g\in H_1^0.
\]

The surface itself is recovered from the orbit structure:
\[
 F_i
 =
 \{x\in X_i\mid H_i^0\cdot x
       \text{ is not open in }X_i\}.
\]
Indeed, the orbit of a point of $F_i$ remains in $F_i$, and hence is
not open in $X_i$, whereas ball-supported isotopies make every orbit
in $X_i\setminus F_i$ open.  Since $h$ conjugates the two actions,
\[
 h\bigl(H_1^0\cdot x\bigr)=H_2^0\cdot h(x).
\]
Thus $h$ preserves whether an orbit is open, and hence $h(F_1)=F_2$.

It remains to recover $\Psi$ on the full group.  Let $f\in H_1$ and
$b\in H_2^0$.  Since $\Psi(H_1^0)=H_2^0$, write
$b=\Psi(a)$ for some $a\in H_1^0$.  The identity component $H_1^0$ is
normal in $H_1$, so $faf^{-1}\in H_1^0$, and hence
\[
 \Psi(f)b\Psi(f)^{-1}
 =
 \Psi(faf^{-1})
 =
 h(faf^{-1})h^{-1}
 =
 (hfh^{-1})b(hfh^{-1})^{-1}.
\]
Thus $\Psi(f)$ and $hfh^{-1}$ induce the same conjugation on $H_2^0$.

Hence
\[
 c=(hfh^{-1})^{-1}\Psi(f)
\]
centralizes $H_2^0$.  We claim that $c$ is trivial.  If
$c\neq\rm{id}$, then, since $X_2\setminus F_2$ is dense, there is a
ball $B\subset X_2\setminus F_2$ such that
$B\cap c(B)=\varnothing$.  A nontrivial element of $H_2^0$ supported
in $B$ cannot then commute with $c$.  Thus
$\Psi(f)=hfh^{-1}$ for every $f\in H_1$.
\end{proof}

\begin{remark}
In both DIFF and TOP categories, the theorem reconstructs the
underlying \emph{unoriented} pair.  The same centralizer argument shows
that the realizing map $h$ is unique.  Chosen orientations of the
ambient manifold or of the surface are additional data not recorded
by these groups.
\end{remark}

Theorem~\ref{thm:pair-reconstruction} recovers the
ambient-equivalence class, but not the isotopy class, of an embedding.
Passing further to the induced mapping classes loses even more
information.  To illustrate these points, for a nonorientable surface $F$, let
$\M(F)=\pi_0(\rm{Diff}(F))$, and define
$\mathcal E^\infty(X,F)$ and $\mathcal E^0(X,F)$ as before.

\begin{prop}
\label{prop:mcg-forgets}
There are both orientable and nonorientable examples of topologically
isotopic but smoothly nonisotopic surfaces in closed simply connected
oriented $4$--manifolds whose smooth and topological extendable mapping
class groups all agree.  In the nonorientable case, there is an
infinite such family whose ambient pairs are pairwise nondiffeomorphic.
\end{prop}

\begin{proof}
The orientable examples are provided by Baraglia
\cite[Theorem~1.3(3)]{Baraglia2024}; see also
\cite[Proposition~2.1]{Auckly2023} for further infinite families.
For a suitable closed simply connected $4$--manifold $X$, Baraglia
constructs a diffeomorphism $f\colon X\to X$ and an embedded sphere
$S\subset X$ such that the spheres
\[
 S_n:=f^n(S),\qquad n\in\Z,
\]
are all topologically isotopic but pairwise smoothly nonisotopic.
Since $f^{m-n}$ carries $S_n$ to $S_m$, the pairs $(X,S_n)$ are all
diffeomorphic, and conjugation identifies their relative
diffeomorphism and homeomorphism groups.  Since
$\M(S^2)=\{1\}$, their smooth and topological extendable mapping class
groups agree vacuously.

For the nonorientable examples, Miyazawa constructs infinitely many
embeddings $\RP\cong P_n\subset S^4$ which are topologically isotopic,
while the pairs $(S^4,P_n)$ are pairwise nondiffeomorphic
\cite[Theorem~1.1]{Miyazawa2023}.  Their topological isotopies identify
the relative homeomorphism groups.  Since all $P_n$ have normal Euler
number $+2$, a diffeomorphism between two such pairs could not reverse
the orientation of $S^4$.  Thus
Theorem~\ref{thm:pair-reconstruction} implies that their relative
diffeomorphism groups are pairwise nonisomorphic.  Nevertheless,
$\M(\RP)=\{1\}$, so $
 \mathcal E^\infty(S^4,P_n)
 =
 \mathcal E^0(S^4,P_n)
 =
 \{1\}
$
for every $n$.
\end{proof}

Thus the  relative groups do not recover smooth isotopy (as expected), while extendable mapping class groups may fail even to distinguish infinitely many pairwise nondiffeomorphic pairs in a single topological isotopy
class.

\bigskip
\noindent {\bf Acknowledgements.}
The first author was partially supported by NSF grant DMS-2506431 and
a Simons Foundation Travel Grant.  Much of this work was completed
during the 2026 IAS/Park City Mathematics Institute program
\emph{Knotted Surfaces in Four-Manifolds}; the authors thank PCMI for
its support and the stimulating environment.  We thank Takuto Sato for
helpful conversations, Bruno Martelli and Leone Slavich for comments on
the hyperbolic constructions in Section~\ref{sec:hyperbolic}, and
Kathryn Mann for comments on the reconstruction results in
Appendix~\ref{appendix:kleinwasright}.  We are especially grateful to
Slavich for confirming the arithmetic embedding argument and
suggesting the refinement recorded in
Remark~\ref{rem:hyperbolic-rigid-recipe}.  We also thank Jin Miyazawa and Danny Ruberman for several helpful comments on an earlier draft, and Hee Jung Kim for clarifications concerning her work cited here.
\newpage
\bibliography{refs1}
\bibliographystyle{plain}

\end{document}